\documentclass[11pt]{article}
\usepackage{amssymb,amsthm, amsmath}

\usepackage{tikz}
\usetikzlibrary{decorations.pathreplacing}
\tikzstyle{snode}=[circle ,draw=black,fill=white,thick, inner sep=0pt ,minimum size=1.7mm]
\tikzstyle{bnode}=[circle ,draw=black,fill=black,thick, inner sep=0pt ,minimum size=1.7mm]

\renewcommand{\baselinestretch}{1.1}
\usepackage[left=2cm,right=2cm,top=1.8cm,bottom=1.8cm]{geometry} 
\newtheorem{theorem}{Theorem}
\newtheorem{claim}{Claim}
\newtheorem{lemma}{Lemma}

\newtheorem{problem}{Problem}

\newtheorem{cor}{Corollary}

\DeclareMathOperator{\lec}{lec}

\DeclareMathOperator{\diam}{diam}
\DeclareMathOperator{\cir}{cir}

\def\ml{l\kern -0.035cm\char 39\kern -0.03cm}

\begin{document}

\title{\bf Loose edge-connection of graphs}
\renewcommand\baselinestretch{1}\small

\author{$^{1}$Christoph Brause, $^2$Stanislav Jendro\ml, $^1$Ingo Schiermeyer
\\ \\
$^1$\small TU Bergakademie Freiberg \\
\small Institute of Discrete Mathematics and Algebra \\
\small 09596 Freiberg, Germany \\
\small \tt E-mail: brause@math.tu-freiberg.de, ingo.schiermeyer@tu-freiberg.de \\
\\
$^2$\small P.J. \v{S}af\'{a}rik University\\ 
\small Jesenn\'a 5, 04001 Ko\v{s}ice, Slovakia\\ 
\small \tt E-mail: stanislav.jendrol@upjs.sk \\
}

\date{\today}
\maketitle

\begin{abstract}
In the last years, connection concepts such as rainbow connection and proper connection appeared in graph theory and obtained a lot of attention. In this paper, we investigate the loose edge-connection of graphs. A connected edge-coloured graph $G$ is loose edge-connected if between any two of its vertices there is a path of length one, or a bi-coloured path of length two, or a path of length at least three with at least three colours used on its edges. The minimum number of colours, used in a loose edge-colouring of $G$, is called the loose edge-connection number and denoted $\lec(G)$. We  determine the precise value of this parameter for any simple graph $G$ of diameter at least 3. We show that deciding, whether $\lec(G) = 2$ for graphs $G$ of diameter 2, is an NP-complete problem. Furthermore, we characterize all complete bipartite graphs $K_{r,s}$ with  $\lec(K_{r,s}) = 2$. 
\medskip

\noindent \textbf{Keywords:} edge colouring, loose edge-connection, $2$-connected graphs, trees
\end{abstract}

\renewcommand\baselinestretch{1.2}\normalsize

\section{Introduction and motivation}

In this paper, we consider simple and undirected graphs only. For notation and graph theoretic terminology, we generally follow the book of West \cite{Wes}. Aside from that, we denote by $n = |V(G)|$ the \textit{order} of a graph $G$ and by $m =|E(G)|$ the \textit{size} of $G$.  The \textit{degree} of a vertex $v$ of $G$, denoted by $\deg_G(v)$, is the number of edges of $G$ incident with the vertex $v$. We will denote by $\Delta(G)$ and by $\delta(G)$ the \textit{maximum degree} and the \textit{minimum degree} of the graph $G$. Furthermore, for simplicity, we denote by $[k,l]$ the set of integers $\{k,k+1,\ldots,l\}$ if $k\leq l$. Note that $[k, l] = \emptyset$ for $k > l$. 

Two edges are adjacent if they share a common vertex. Two edges are 
\textit{semi-adjacent} if they are not adjacent but there is a third edge that is adjacent to both of them.

The \textit{weight} $w(e)$ of an edge $e=uv$ is the sum of degrees of its end-vertices, i.e. $w(e)=deg(u)+deg(v)$.
The \textit{reduced maximum weight} $\mathrm{rw}(G)$ of the graph $G$ is defined as $\mathrm{rw}(G)=max\{w(e)-1:e \in E(G)\}$. 

Let $G$ be a graph. A \emph{path} $L$ in $G$ is a subgraph of $G$ consisting of a vertex set $\{v_1,v_2,\ldots,v_k\}$ and edge set $\{v_1v_2,v_2v_3,\ldots, v_{k-1}v_k\}$. For simplicity, we write $L:v_1v_2\ldots v_k$ and, if $L$ contains two vertices $u,v$, then we denote by $uLv$ the subpath of $L$ between $u$ and $v$. Two vertices, say $u,v\in V(G)$, are \emph{connected} by a path $L$ if the end vertices of $L$ are $u$ and $v$. Such a path is denoted as an $u, v$-\textit{path}.

The \textit{circumference} of a graph $G$, $\mathrm{cir}(G)$, is the length of the longest cycle. We will denote by $C_k$ (resp. $K_k$, resp. $K_{r,s}$) the cycle (resp. the complete graph, resp. complete bipartite graph) on $k \geq 3$ (resp. $k$, resp. $r \geq 1$ and $s \geq 1$) vertices.

Let $\mathbb{A}=\{a,b,c,\ldots\}$ be a finite alphabet, i.e. a set of colours, digits, symbols, \ldots, whose elements are called \emph{letters}. A \emph{word} $A$ of length $k$ over $\mathbb{A}$ is a sequence of letters, say $A:a_1a_2a_3\ldots a_k$ where $a_i\in \mathbb{A}$ for all $i\in[1,k]$. Let us recall some properties that words can have (see e.g. \cite{CJ}).
A word $A$ is 
\begin{description}\itemsep-2pt
\item[\ldots\ \it{proper}] if  consecutive letters are not identical,
\item[\ldots\ \it{rainbow}] if it does not contain two identical letters,
\item[\ldots\ \it{conflict-free}] if at least one letter occurs exactly once in $A$,
\item[\ldots\ \it{monochromatic}] if all letters are identical,
\item[\ldots\ \it{odd}] if each letter of the alphabet $\mathbb{A}$ appears an odd number of times or zero times in $A$,
\item[\ldots\ \it{loose}] if either $A$ consists of exactly one letter, of exactly two different letters, or has at least three different letters,
\end{description}

Using properties of words over an alphabet $\mathbb{A}$, Brause, Jendro\ml, and Schiermeyer \cite{BJSch} have introduced a graph theoretic meta-concept as follows.

Consider a graph $G$. If there is an edge colouring $\phi:E(G)\to \mathbb{A}$, then we can associate to a path $L = v_1v_2\ldots v_{k}$ with $v_i\in V(G)$ for all $i\in[1,k]$, $L$ being a subgraph of $G$, a word $\phi(v_1v_2)\phi(v_2v_3)\ldots \phi(v_{k-1}v_k)$.

Now, let $\mathcal{P}$ be a property of words over the alphabet $\mathbb{A}$. 
Considering an edge colouring $\phi:E(G)\to \mathbb{A}$ we say that \emph{$L$ has the property $\mathcal{P}$} if the associated word $\phi(v_1v_2)\phi(v_2v_3)\ldots \phi(v_{k-1}v_k)$ has property $\mathcal{P}$.

Let $G$ be a connected graph, $\mathbb{A}$ be an alphabet with $k$ letters, $\mathcal{P}$ be a property of words, and $\phi:E(G)\rightarrow\mathbb{A}$ be an edge-colouring. The edge-colouring \emph{$\phi$} makes $G$ to be $\mathcal{P}$ edge-connected if any two distinct vertices of $G$ are connected by a path having the property $\mathcal{P}$. The minimum integer $k$, for which there exists an edge colouring $\phi:E(G)\to \mathbb{A}$ with $|\mathbb{A}|=k$ that makes $G$ to be $\mathcal{P}$ edge-connected, is the $\mathcal{P}$ \textit{edge-connection number} of $G$.

From a practical point of view, the $\mathcal{P}$ edge-connection of graphs plays an important role for security and accessibility in communication networks. While the information which one sends through a network from one node to another has to be protected by passwords, it is of high importance that the password sequences of information transferring paths meet some prescribed requirements. Since managing a whole bunch of passwords, that are assigned to direct information transferring paths between two nodes, is expensive, it is a natural question to ask for a minimum number of passwords securing the information transferring paths.
By representing each node as a vertex, each direct information transferring path between two nodes as an edge, each possible password by a different colour, and the prescribed requirements for the password sequences of information transferring paths by a property $\mathcal{P}$ for a word over the alphabet of colours, the above problem is translated to finding the $\mathcal{P}$ edge-connection number of a path.

Chartrand et al. \cite{CJMZ} introduced the concept of rainbow connection; that is, the concept of the $\mathcal{P}$ edge-connection where property $\mathcal{P}$ is rainbow. Let $\mathrm{rc}(G)$ denote the \textit{rainbow edge-connection number}.
Returning to the mathematical problem and its applications, the $\mathcal{P}$ edge-connection number of graphs have been intensively studied for some of the mentioned properties, for example if property $\mathcal{P}$ is rainbow \cite{Chan,CJMZ,HRS,HRSV1,HRSV2,LSS}, proper \cite{ABBFKS,borozan12}, odd \cite{BJSch}, or conflict-free \cite{Czap}. In this paper, we concentrate on the loose edge-connection, which is a relaxation of the concepts rainbow connection and conflict-free connection. A motivation to study the loose edge-connection comes also from the loose-colouring of embedded graphs, see e.g. \cite{Ne:05}, \cite{Cz:12}, or \cite{CJ}. Note, that the loose vertex-connection of simple connected graphs is studied in a more general version in \cite{BJSch-21}. 

We say that a path $L$, which is a subgraph of a graph $G$ with an edge colouring $\phi:E(G)\to [1, k]$, is \emph{loose edge-coloured} if the path consists of one edge, or exactly two differently coloured edges, or has at least three edges that are coloured with at least three distinct colours.

A connected graph $G$, coloured by an edge colouring $\phi:E(G)\rightarrow [1,k]$, is said to be \emph{loose edge-connected} if any two distinct vertices are connected by a loose coloured path.
The least integer $k$, for which we have an edge colouring $\phi:E(G)\to [1, k]$ that makes $G$ loose edge-connected, is called the \emph{loose edge-connection number} of $G$, denoted by $\lec(G)$.

Observe, that any rainbow connected graph is also loose connected. Also, if a graph $G$ is a spanning subgraph of a graph $H$ and  $\lec(G) = k$, then  $\lec(H) \leq k$.

In the sequel, a path in an edge coloured graph will be called \emph{monochromatic}, \emph{bi-chromatic}, and \emph{loose} if its edges are coloured with exactly one colour, exactly two colours, and with at least three colours, respectively.

A \emph{block} $B$ of a connected graph $G$ is an induced subgraph of $G$ of order at least $2$ which is $2$-connected and, with respect to this condition, maximal; that is, $B$ is $2$-connected, but for every $v\in V(G)\setminus V(B)$, the graph $G[V(B)\cup\{v\}]$ is disconnected or contains a cut-vertex.
We note that if $n(B) \geq 3$, then we say that $B$ is \emph{non-trivial}. Otherwise, if $n(B) = 2$ and $B$ is isomorphic to $K_2$, we say that $B$ is \emph{trivial}.

The \textit{block-cutpoint graph} of a graph $G$ is a bipartite graph $B(G)$ in which one partite set consists
of the cut-vertices of $G$, and the other has a vertex $b_i$ for each block $B_i$ of $G$. We include $vb_i$ as an edge of $B(G)$ if and only if $v \in B_i$. When $G$ is connected, its block-cutpoint graph is a tree whose leaves are blocks of $G$. These blocks are called \emph{leaf blocks}.
If $B$ is a trivial block and $e\in E(B)$ is the edge of $B$, then $e$ is a \emph{cut-edge} of $G$; that is, $G-e$ is disconnected.

Let $C(G)$ be the graph induced on the set $E(C)\subseteq E(G)$ of all cut-edges of $G$ consisting of all vertices of $G$ that are incident to the edges of the set $E(C)$. More precisely, 
$$V(C(G))=\left\{\bigcup_{uv\in E(C)}\{u,v\}\right\}\quad{\rm and}\quad E(C(G)) = E(C).$$
We denote by $\Delta(C(G))$ the maximum degree of $C(G)$. The graph $C(G)$ is called the \emph{cut-edge graph} of $G$.
 (see \cite{Wes}, p. 156).

Let $F_t$ be a graph obtained from a graph $F$ by attaching $t$ leaves, $t \geq 0$, to every vertex of $F$.

Let $H$ be a non-trivial connected graph, $F$ be a $2$-connected graph, and $t \geq 0$ be an integer. A graph $H$ is \textit{of type (t, F)} if $F \subseteq H \subseteq F_t$, and $H$ contains a vertex adjacent with exactly $t$ leaves.

The rest of this paper is organized as follows: Sections 2 - 7 contain some preliminary statements that are necessary for the proofs of our main results. In Sections 8 the characterization of all complete bipartite graphs with $\lec (K_{r,s}) = 2$ is given. In Section 10, the precise values of the loose edge-connection number for all non-trivial connected graphs of diameter at least 3 are determined . 


\section{Loose edge-connection of trees}

\begin{lemma}\label{le:tre}
If $G$ is a tree of the reduced weight $\mathrm{rw}(G)$, then  $\lec (G) = \mathrm{rw}(G)$.
\end{lemma}
\begin{proof} For the edge-colouring of $G$ we use the following adaptation of the FSB algorithm (see e. g. \cite{Wes}, p. 99). 

For a tree $G$ let $M = [1, \mathrm{rw}(G)]$ be the set of colours. 

\textbf{Input:} A tree $G$, a start vertex $u$, and the colour set $M$.

\textbf{Idea:} Maintain a set $R$ of vertices that have been reached by coloured edges but not searched and a set $S$ of vertices that have been searched, i.e. vertices of the set $S$ are incident only with coloured edges.
The set $R$ is maintained as a First-In First-Out list (queue), so the first vertices are the first vertices explored.

\textbf{Initialization:}  $R = \{u\}, S = \emptyset$.

\textbf{Iteration:} As long as $R \neq \emptyset$, we search for the first vertex $v$ of $R$. Then we colour each edge $e = vw$ of $G$ incident with $v$ and 
$w \notin R \cup S$, with colour $\varphi(e)=\min\{M \setminus F(e)\}$. (Here  set $F(e)$ is the set of colours used till now on edges adjacent or semi-adjacent with $e$.) Next the neighbour $w$ of $v$ is added to the back of $R$. If all the edges incident with $v$ are coloured, the vertex $v$ is removed from the front of $R$ and is placed in $S$.

It is easy to see that this algorithm works properly.
\end{proof}

\section{Graphs of type $R$}

A graph $H$ is of type $(t, R)$  if $R = C_3 = K_3$.
Then $R_t$ is the graph obtained from $K_3$ by attaching $t$ leaves on every vertex of it. Let $V(R_t) = \{x,y,z,x_i,y_i,z_i; i\in[1,t]\}$, $E(R_t) = \{xy, xz, zy, zz_i, yy_i, xx_i; i\in[1,t]\}$. Observe that $\Delta(C(R_t)) = \mathrm{rw}(C(R_t))=t$.

\begin{lemma}\label{le:R_t}
If $H$ is a graph of type $(t, R)$, then the following hold: 
\begin{enumerate} 
\item
If $t = 1$ and $|E(C(H))| = 1$, then $\lec(H) = 2$ .
\item 
If $t \leq 2$ and $|E(C(H))| \geq 2$, then  $\lec(H) = 3$.
\item
If $t \geq 3$, then $t \leq \lec(H) \leq t + 1$.

Moreover, $\lec(H) = t+1$ if and only if $H$ contains two vertices $b_1$ and $b_2$ such that $\deg_{C(H)}(b_1) + deg_{C(H)}(b_2) \geq 2t - 1$. 
\end{enumerate}
\end{lemma}

\begin{proof} 
\textbf{Case 1}. First we consider the graph $R_t$.
Observe that $\diam(R_t) = 3$ and $\lec(R_t) \geq \max\{3, t\}$. 
The following loose edge-colouring with three colours is suitable for $R_1$:
$\varphi(xy) = \varphi(zz_1) = 1$,  $\varphi(yz) = \varphi(xx_1) = 2$, and  $\varphi(xz) = \varphi(yy_1) = 3$.

The colouring $\varphi(xx_i)=\varphi(yy_i)=\varphi(zz_i)=i+1$ for every $i\in[3,t]$, $\varphi(xy) = \varphi(xx_1) = \varphi(yy_1) = 1$, $\varphi(xz) = \varphi(xx_2) = \varphi(zz_1) = 2$, and $\varphi(yz) = \varphi(yy_2) = \varphi(zz_2) = 3$ is a suitable loose edge-colouring of $R_t$ with $t + 1$ colours for all $t \geq 2$.

We need to show that $t + 1$ colours are necessary for $t \geq 3$. We start with the graph $R_3$. Assume that three colours are enough to get a loose connection of $R_3$. Then all three leaves attached to each vertex of $C_3$ of $R_3$ are coloured with three different colours. Consider the vertices $x$ and $y$. Let, w.l.o.g.,
$\varphi(xx_i)=\varphi(yy_i)=\varphi(zz_i)=i, i \in [1, 3]$. There is a loose $x_i, y_i$-path in $R_3$ such that $\{\varphi(xz), \varphi(yz)\} = [1, 3] \setminus \{i\}$.
But for any $j, j \neq i$, we also have 
$\{\varphi(xz), \varphi(yz)\} = [1, 3] \setminus \{j\}$.
A contradiction.

Assume that there is a loose colouring of $R_t$ with $t$ colours for $t \geq 4$, 
$\varphi(xx_i)=\varphi(yy_i)=\varphi(zz_i)=i, i \in [1, t]$, and let colours $1, 2$, and $3$ are used on the edges of $C_3$. If we delete from $R_t$ the edges coloured with colour $k \in [4, t]$, we get a loose colouring of $R_3$ with three colours. This contradicts the nonexistence of such a colouring of $R_3$.

 \textbf{Case 2.}
 Let $t\geq 3$. Consider first the graph $G_1 = R_t \setminus \{xx_t, xx_{t-1}, yy_t, yy_{t-1}\}$. It has the following loose edge-colouring $\varphi$ with $t$ colours:
 $\varphi(xx_i) = \varphi(yy_i) = \varphi(zz_i) = i$, $i \in [1, t - 3]$, 
$\varphi(xx_{t-2}) = \varphi(xz) = \varphi(zz_t) = t$, 
 $\varphi(yy_{t-2}) = \varphi(yz) = \varphi(zz_{t-1}) = t - 1$, and 
 $\varphi(xy) = \varphi(zz_{t-2}) = t - 2$.
 It is easy to see that any graph $H$ of type $(t, R)$ which is a subgraph of $G_1$ has a loose edge-colouring with $t$ colours.
 
 We need to show that $t + 1$ colours are necessary for $t \geq 3$ if $H$ has two vertices, say, $x$ and $z$, whose degree sum in $C(H)$ is at least $2t - 1$.
 Assume for the contrary that the graph $H$ has a loose edge-colouring $\varphi$ with $t$ colours.
 Consider the graph $H$ and its vertices $x, x_i,i \in [1, t - 1]$, $z$, and $z_j, j \in [1, t]$. Let, w.l.o.g., $\varphi(xx_i) = \varphi(zz_i) = i,i \in [1, t - 1]$, and $\varphi(zz_t) = t$. Let $\varphi(xz) = c$. Then the edges $xy$ and $yz$  have to be coloured with different colours, say $a$ and $b$. Then there is neither a loose $x_a, z_a$-path nor a loose $x_b, z_b$-path in $H$. 
 A contradiction.

\textbf{Case 3.}
Let $t \leq 2$. If $H$ contains at least three leaves, then  $\diam(H) = 3$, $H \subseteq R_2$, and, by Case 1, $\lec(H) = 3$. If $H$ has exactly two leaves, then one can easily show that $\lec(H) = 3$ and an edge-colouring of $C_3$ consists of three different colours. If $H$ has one leaf, we colour all the edges of the $3$-cycle of $C_3$ with colour $1$ and remaining edges of $H$with colour $2$.

This completes the proof of our Lemma \ref{le:R_t}.
\end{proof}

\section{Graphs of type $Q$}


Let $Q_t$ be a graph obtained from the cycle $C_4$ by attaching $t$ leaves, $t \geq 1$, to every vertex of $C_4$. Let 
$V(Q_t) = \{x, y, v, w, x_i, y_i, v_i, w_i; i \in [1, t]\}$  and $E(Q_t) = \{vx, wx, vy, wy, xx_i, yy_i, vv_i, ww_i; i \in [1, t]\}$.

\begin{lemma}\label{le:Q_t}
If $H$ is a graph of type $(t, Q)$ with $C_4 \subseteq Q \subseteq K_4$, then the following hold:
\begin{enumerate} 
\item[\rm (i)] If $1 \leq t \leq 2$, then $3 \leq \lec(H) \leq 4$. Moreover, $\lec(H) = 4$ if and only if the cycle $C_4$ of $H$ contains no diagonal and at least three vertices each adjacent to exactly two leaves.
\item[\rm (ii)] If $t = 3$, then $3 \leq \lec(H) \leq 4$. Moreover, $\lec(H) = 4$ if and only if the cycle $C_4$ of $H$ does not contain any diagonal, or $C_4$ contains a diagonal $xy$, $\deg_{C(H)}(x) = 3$ and $\deg_{C(H)}(y) = 3$, or $\deg_{C(H)}(y) = 2$ and $\deg_{C(H)}(v) = \deg_{C(H)}(w) = 3$.
\item[\rm (iii)] If $t \geq 4$, then $\lec(H) = t$ .
\end{enumerate}
\end{lemma}

\begin{proof}
\textbf{Case 1.}
First we consider the graph $Q_t$.
Observe that $\diam(Q_t) = 4$ and $\lec(Q_t) \geq \max\{3, t\}$. 
 A suitable loose edge-colouring of $Q_1$ with three colours is:
$\varphi(vx) = \varphi(wy) = a, \varphi(vy) = c$, $\varphi(wx) = b$ and $\varphi(xx_1) = \varphi(yy_1) = c, \varphi(vv_1) = \varphi(ww_1) = b$.

A suitable loose edge-colouring of $Q_t, t \geq 2$, with at least four colours follows:
$\varphi(vx) = a, \varphi(wx) = b, \varphi(wy) = c, \varphi(vy) = d$, and  $\varphi(xx_i) = \varphi(yy_i) = \varphi(vv_i) = \varphi(ww_i) = i, i \in [1, t]$. Here $\{a, b, c, d\} \subseteq [1, t]$. This gives $\lec(Q_2) \leq 4$, $\lec(Q_3) \leq 4$, and $\lec(Q_t) = t$ for $t \geq 4$.

\textbf{Case 2.}
 If $t \geq 4$, then the loose colouring of $H$ with $\max\{3, t\}$ colours follows from the loose colouring of $Q_t$ with $t$ colours. 

\textbf{Case 3.}
Let $t = 2$.

\textbf{Subcase 3.1}
Let on $C_4$ of $H$ there are at most two vertices each adjacent with exactly two leaves. Then it is a subgraph of one of the following graphs: $H_1 = (V_1, E_1)$ with $V_1 = \{x, x_1, x_2, v, v_1, v_2, y, y_1, w, w_1\}$, $E_1 = \{xv, xx_1, xx_2, vv_1, vv_2, vy, yy_1, yw, ww_1, wx\}$ and a loose edge-colouring $\phi(xx_1) = \phi(vv_1) = \phi(vy) = \phi(yy_1) = a$, $\phi(xv) = \phi(yw) = b$, $\phi(xx_2) = \phi(wx) = \phi(vv_2) = \phi(ww_1) = c$. $H_2 = (V_2, E_2)$ with $V_2 = \{x, x_1, x_2, v, v_1, y, y_1, y_2, w, w_1\}$, $E_2 = \{xv, xx_1, xx_2, vv_1, vy, yy_1, yy_2, yw, ww_1, wx\}$ and a loose edge-colouring $\phi(xx_1) = \phi(vv_1) = \phi(vy) = \phi(yy_1) = a$, $\phi(xv) = \phi(yw) = c$, $\phi(xx_2) = \phi(wx) = \phi(yy_2) = \phi(ww_1) = b$.

It is easy to see that the graph $H$ has a loose edge-colouring with three colours obtained from one of the above colouring of the graph $H_1$ or $H_2$. 

\textbf{Subcase 3.2}
Let on $C_4$ of $H$ there are at least three vertices, say $x, v$, and $y$, each adjacent with exactly two leaves. Assume that there is a loose edge-colouring $\phi$ with three colours. We can assume, w.l.o.g., that 
$\phi(xx_1) = \phi(vv_1) = a$ and $\phi(xx_2) = b$.

\textbf{Subcase 3.2.1} Let $\phi(xv) = a$. Then $\{\phi(xw), \phi(wy)\} = \{b, c\}$, otherwise there is no loose $x_1, y$-path, and $\{\phi(vy), \phi(wy)\} = \{b, c\}$, otherwise there is no loose $v_1, w$-path in $H$. This gives, in order, $\phi(vv_2) = c$ (because a loose $x_2, v_2$-path), $\phi(xw) = b$, $\phi(wy) = c$, $\phi(vy) = c$ (because a loose $x_2, y$-path) and $\phi(wy) = b$ (because a loose $v_1, w$-path), a contradiction. 

\textbf{Subcase 3.2.2}  Let $\phi(xv) = b$. Then $\{\phi(xw), \phi(wy)\} = \{a, c\}$ and $\phi(vy) = c$ because of a loose $x_2, y$-path and loose $x_1, y$-path. This implies that there is no $x_1, v_1$-path, a contradiction.

\textbf{Subcase 3.2.3} Let $\phi(xv) = c$. Then $\{\phi(xw), \phi(vy)\} = \{a, b\}$, $\phi(wy) = c$ $\phi(vv_2) = b$, and, w.l.o.g., $\phi(yy_1) = a$, $\phi(yy_2) = b$. This implies the nonexistence a loose $v_i, y_i$-path in $H$ for some $i \in \{1, 2\}$, a contradiction.

\textbf{Subcase 3.3}
Let the cycle $C_4$ of $H$ contains a diagonal $xy$, that is $H$ contains a block which is the diamond. Then $H$ is a subgraph of the graph $D_2 \cup \{xy\}$. This graph has the following loose edge-colouring $\phi$. $\phi(xx_1) = \phi(vv_1) = \phi(yy_1) = \phi(vy) = a$, $\phi(xw) = \phi(xx_2) = \phi(yy_2) = \phi(ww_1) = b$, and $\phi(xv) = \phi(xy) = \phi(vv_2) = \phi(wy) = \phi(ww_2) = c$.
From this colouring we can easily get a required colouring of $H$.

\textbf{Case 4.}
Let $t = 3$ and $H$ does not contain a diagonal $xy$. Assume, that $H$ has a loose edge-colouring with tree colours and $\deg_{C(H)}(x) = 3$. Let the neighbour leaf $x_i$ of $x$ be adjacent to $x$ via edge $xx_i$ coloured with a colour $i$. Then between $x$ and $y$ there should be three mutually different bi-chromatic $x,y$-paths but there exactly two ones. A contradiction, which shows that at least four colour are necessary for a loose edge-colouring of $H$. Such a colouring is induced by the colouring of $Q_3$ of the Case 1 because $H \subseteq Q_3$. 

\textbf{Case 5.}
Let $t = 3$ and $H$ contain a diagonal $xy$. Denote by $D$ the diamond, that is the subgraph of $H$ induced on the vertex set $\{x, v, y, w\}$. Assume, that $H$ has a loose edge-colouring $\varphi$ with tree colours.

\textbf{Subcase 5.1.} Let $\deg_{C(H)}(x) = 3$. Let the neighbour leaf $x_i$ of $x$ be adjacent to $x$ via edge $xx_i$ coloured with a colour $i$.

\textbf{Subcase 5.1.1.} Let $\deg_{C(H)}(y) = 3$ and let the neighbour leaf $y_i$ of $y$ is adjacent to $y$ via the edge $yy_i$ coloured with colour $i, i \in \{1, 2, 3\}$. As $H$ has a loose edge-colouring with three colours, there should be three distinct bi-chromatic $x,y$-paths of length two. But there are exactly two, a contradiction.

\textbf{Subcase 5.1.2.} Let $\deg_{C(H)}(y) = 2$
Let $\bar Q_3 = (Q_3 \setminus \{yy_1, ww_3\}) \cup \{xy\}$.
The graph $\bar Q_3$ has the following loose edge-colouring:
$\varphi(vx) = \varphi(wx) = 1$, $\varphi(xx_1) = \varphi(vv_1) = \varphi(ww_1) = 1$; $\varphi(yw) = \varphi(yy_2) = 2$, $\varphi(xx_2) = \varphi(vv_2) = \varphi(ww_2) = 2$, and $\varphi(vy) = \varphi(xy) = \varphi(xx_3) = 3$, $\varphi(yy_3) = \varphi(vv_3) = 3$. 
Observe, that each subgraph $H \subseteq \bar Q_3$ of type $(t, Q)$ containing the diamond $D$ has a loose edge-colouring with three colours. 

To the end of the proof in this case it is enough to show that the graph $Q^* = (Q_3 \setminus \{yy_3\}) \cup \{xy\}$
 does not have any loose edge-colouring with three colours. Assume the opposite. Let it has such a colouring $\varphi$.

Let the neighbour leaf $y_i$ of $y$ be adjacent to $y$ via the edge $yy_i$ coloured with a colour $i \in \{1, 2\}$.
Then the following situation, w.l.o.g., appears in $Q^*$: 
$\{\varphi(xv), \varphi(vy)\} = \{1, 3\}$ and $\{\varphi(xw), \varphi(wy)\} = \{2, 3\}$ because there are loose $x_1, y_1$-path and $x_2, y_2$-path of length al least three. This, w.l.o.g.,  enforces $\varphi(xy) = 1$ and then no loose $x_3,v_3$-path exists. A contradiction

\textbf{Subcase 5.1.3.}
Let $\deg_{C(H)}(y) \leq 1$ and let $y_1$ be the neighbour of $y$. Then $H \subseteq G^1 = (Q_3 \setminus \{y_2, y_3\}) \cup \{xy\}$. The graph $G^1$ has the following loose edge-colouring with three colours:
$\varphi(xv) = \varphi(xw) =$ $\varphi(xx_1) =$ $\varphi(ww_1) = \varphi(vv_1) = 1$, 
$\varphi(xy) = \varphi(ww_2) =$ $\varphi(xx_2) =$ $\varphi(yw) = \varphi(vv_2) = 2$, and
$\varphi(yv) = \varphi(ww_3) =$ $\varphi(xx_3) =$ $\varphi(yy_1) = \varphi(vv_3) = 3$.
It is easy to see that colouring $\varphi$ of $G^1$ just described has the required properties of loose edge-coloring.
From this colouring of $G^1$ one easily derives a loose edge-colouring with three colours of any subgraph $H$ of $G^1$, which is of of type $(t, Q)$.

\textbf{Subcase 5.2.} Let $\deg_{C(H)}(x) \leq 2$ and $\deg_{C(H)}(y) \leq 2$. Consider the graph 
$G^* = Q_3 \setminus \{xx_2, yy_1\} \cup \{xy\}$. The following edge-colouring $\varphi$ of $G^*$ is a loose one: $\varphi(xv) = \varphi(xw) =\varphi(xx_1) = 1$, $\varphi(vv_1) = \varphi(ww_1) = 1$, 
$\varphi(vy) = \varphi(wy) = \varphi(vv_2) = 2$,
$\varphi(yy_2) = \varphi(ww_2) = 2$, 
$\varphi(xy) = \varphi(xx_3) = \varphi(vv_3) = 3$, and $\varphi(yy_3) = \varphi(ww_3) = 3$.
From this colouring of $G^*$ one easily derives a loose edge-colouring with three colours of any subgraph $H$ of $G^*$, which is of type $(t, Q)$.

\textbf{Case 6.}
Let the graph $H$ contain both diagonals $xy$ and $vw$.
Then $H \subseteq G^2 = Q_3 \cup \{xy, vw\}$. It is easy to find a suitable loose edge-colouring of $G^2$ with three colours when we color edges $xv$ and $yw$ with colour $1$, the edges $xw$ and $yv$ with colour $2$, and the edges $xy$ and $vw$ with colour $3$.

This finishes the proof of Lemma \ref{le:Q_t}
\end{proof}

\section{Graphs of type $P$}


Let $P_t$ be a graph obtained from the cycle $C_5$ by attaching $t$ leaves, $t \geq 1$, to every vertex of $C_5$. Let 
$V(P_t) = \{x, y, z, u, v, x_i, y_i, z_i, u_i, v_i; i \in [1, t]\}$  and $E(P_t) = \{ux, vx, uy, vz, yz, xx_i, yy_i, zz_i, uu_i, vv_i; i \in [1, t]\}$.

\begin{lemma}\label{le:P_t}
If $H$ is of type $(t, P)$ with $C_5 \subseteq P \subseteq K_5$, then the following statements hold:
\begin{enumerate} 
\item[\rm (i)] If $1 \leq t \leq 2$, then $\lec(H) = 3$ .
\item[\rm (ii)] If $t = 3$, then $3 \leq \lec(H) \leq 4$. 
Moreover $\lec(H) = 4$ if and only if the cycle $C_5$ of $H$ does not contain any diagonal and there are, on $C_5$, two vertices each of which is adjacent to three leaves. 

\item[\rm (iii)] If $t \geq 4$, then $\lec(H) = t$ .
\end{enumerate}
\end{lemma}

\begin{proof}
\textbf{Case 1.}
First we consider the graph $P_t$.
As $\diam(P_t) \geq 3$, we have $\lec(P_t) \geq \max\{3, t\}$.
A suitable loose colouring of $P_1$ is as follows:
$\varphi(xv) = \varphi(yz) = \varphi(xx_1) = \varphi(vv_1) = b$, $ \varphi(xu) = \varphi(vz) = \varphi(zz_1) = a$, and 
$\varphi(uy) = \varphi(uu_1) = \varphi(yy_1) = c$.

To get a suitable loose edge-colouring of $P_2$ we first take a subgraph $P_1$ of $P_2$ and colour its edges as above and then colour the remaining edges of $P_2$ as follows:
$\varphi(uu_2) = a,  \varphi(yy_2) = b$, and $\varphi(xx_2) = \varphi(vv_2) = \varphi(zz_2) = c$.

There is no loose edge-colouring of $P_3$ with three colours. Assume, for a contrary, that there is one, $\phi$, with colours $a, b$, and $c$. Observe, that there must be an edge, say, w.l.o.g., $uy$, with $\phi(uy) = b$ and $\phi(e) \ne b$ for all $e \in \{ux, vx, vz, yz\}$. As all three colours have to be used on the the edges incident to leaves adjacent to $u$ as well as to $y$, there are two edges $uu_i$ and $yy_j$ with $\phi(uu_i) = \phi(yy_j) = a \neq b$. Then there is no loose $u_i,y_j$-path, a contradiction.

The following is a suitable loose colouring of $P_t, t \geq 3$:
$\varphi(xu) = a, \varphi(xv) = \varphi(yz) = c, \varphi(uv) = b, \varphi(vz) = d$, and $\varphi(xx_i) = \varphi(yy_i) = \varphi(zz_i) = \varphi(uu_i) = \varphi(vv_i) = i, i \in [1, t]$. Here $\{a, b, c, d\} \subseteq [1, t]$. This gives $\lec(P_3) = 4$ and $\lec(P_t) = t$ for $t \geq 4$. It is easy to see that any graph $H \subseteq P_t$, which is of type (t, P) for $t \geq 4$, has $\lec(H) = t$.

\textbf{Case 2.}
Let $t = 3$, $H$ contain no diagonal on its cycle $C_5$, and have at least two vertices adjacent to exactly three leaves. Assume that $H$ has a loose edge-colouring $\phi$ with three colours $1, 2$, and $3$. Evidently, all these three colours have to appear on $C_5$ and let, w.l.o.g., colour $3$ be unique on $C_5$. 

\textbf{Subcase 2.1.} The graph $H$ contains on $C_5$ two adjacent vertices, say $u$ and $x$, each of which is adjacent to exactly $3$ leaves. Let $\phi(xx_i) = \phi(uu_i) = i$. 
Colour $3$ cannot appear on the edge $ux$ because otherwise there is no loose $u_1,x_1$-path and we would have a contradiction.
Let $a \in \{1, 2\}$. 
If $\{\phi(xv), \phi(vz)\} = \{a, 3\}$ then there is no loose $x_a,z$-path in $H$; a contradiction. The case $\{\phi(xu), \phi(uy)\} = \{a, 3\}$ is analogous. Hence, in this case, $H$ requires four colours for its loose edge-colouring. As $H \subseteq P_3$ we can use the loose edge-colouring of $P_3$ for $H$. 

\textbf{Subcase 2.2.} The graph $H$ contains on $C_5$ two non-adjacent vertices, say $x$ and $y$, each of which is adjacent to exactly $3$ leaves. Let $\phi(xx_i) = \phi(yy_i) = i$. If $\{\phi(ux), \phi(uy)\} = \{a, 3\}$ then there is no loose $x_a,y_a$-path in $H$; a contradiction.
If $\{\phi(xv), \phi(vz)\} = \{a, 3\}$ then there is no loose $x_a,z$-path in $H$; a contradiction. The case when $\{\phi(yz), \phi(vz)\} = \{a, 3\}$ is symmetric to the latter one. This means that $H$ requires four colours for its loose edge-colouring. Now we use the loose edge-colouring of $P_3$ for $H$. 

\textbf{Case 2.3.}
Let $H$ contain on $C_5$ exactly one vertex, say $v$, adjacent to exactly three leaves. Then $H$ is subgraph of the graph $ G^3 = P_2 \cup \{v_3\} \cup \{vv_3\}$, which has a loose edge-colouring obtained from the colouring of $P_2$ extended by $\phi(vv_3) = a$. This colouring of $G^3$ induces a loose edge-colouring of $H$ with three colours.

\textbf{Case 3.}
Let $H$ contain the cycle $C_5$ with a diagonal. Then, w.l.o.g., $H$ contains a subgraph $G^4 = (V(G^4), E(G^4))$ with vertex set $V(G^4) = \{x, y, z, u, v\}$ and the edge set $E(G^4) = \{ux, vx, uy, yz, xy, vz\}$. If we colour the edges of $E(G^4)$ as follows $\phi(vx) = \phi(xy) = 1$, $\phi(ux) = \phi(yz) = 2$, and $\phi(uy) = \phi(vz) = 3$, we can easily extend this colouring to a loose edge-colouring of $H$ with three colours.

\textbf{Case 4.} If $H$ contains at least two diagonals on its $C_5$, then we solve first the situation for its subgraph with one diagonal as in Case 3, and then colour arbitrarily the remaining diagonals. 

This completes the proof of our Lemma \ref{le:P_t}.
\end{proof}

\bigskip
For the completeness we mention the following:
\begin{lemma}
If $G_t$ is a graph obtained from the cycle $C_k, k \geq 6$, by attaching $t$ leaves, $t \geq 1$, to every vertex of $C_k$, then $\lec(G_t) = \max\{3, t\}$.
\end{lemma}


\section{2-connected graphs with small circumference}

The smallest size among connected graphs $G$ with $\mathrm{cir}(G) = k$ has the cycle $C_k$. The problem of characterization of all ($2$-connected) graphs $G$ with property 
$\mathrm{cir}(G) = k$ is very difficult in general but it is very interesting. As an example of a specific problem of this type is the problem of characterization of Hamiltonian graphs, see e.g. \cite{Wes}. 
There are also lot of papers devoted to the study of estimations on bounds of the circumferences of graphs from particular classes.  
In this section we characterize all $2$-connected graphs of small circumference.

Denote by $K_{2,s}, s \geq 2$, the bipartite graph with vertex set $V(K_{2,s})=\{x,y, v_i, i \in[1,s]\}$ and edge set $E(K_{2,s})=\{xv_i, yv_i, i \in [1, s]\}$. 

The graph $K'_{2, s}$ is obtained from  $K_{2,s}$ by inserting the edge $xy$. ($D = K'_{2, 2}$ is also known as a \textit{diamond}.)
The graph $K^+_{2,s}$, $s \geq 3$, is obtained from  $K_{2,s}$ by inserting the edge $v_2v_3$. 

Let $P_{r,s}$, $r\geq 1, s\geq 1$, be the graph with vertex set $V(P_{r,s})$ = $\{x,y,z,u_1,...,u_r,v_1,...,v_s\}$
and edge set $E(P_{r,s}) = \{yz, xu_i, yu_i, i\in [1,r], xv_j, zv_j, j \in [1,s]\}$.

The graphs $P'_{r,s}$ and  $P''_{r,s}$, $r \geq 1, s \geq 1$, come from the graph $P_{r,s}$ by inserting the edge 
$xy$ and the edges $xy$ and $xz$, respectively. 

Note that $C_5 = P_{1, 1}, C^1_5 = P'_{1, 1}, 
C^2_5 = P''_{1, 1}$, $\bar C^2_5 = P'_{1, 1} + u_1v_1$, $C^3_5 = P''_{1, 1} + yv_1$, $\bar C^3_5 =\bar P''_{1, 1} + u_1v_1$, $C^4_5 = C^3_5 +yv_1$, $C^5_5 = K_5$, and all denote the cycle $C_5$ of length 5 with $k$ diagonals, $k \in [1, 5]$.

\begin{theorem}\label{thm:1}
Let $\mathcal{B}_k$ be the set of all $2$-connected graphs $G$ with $\mathrm{cir}(G) = k$. Then
\begin{enumerate}
\item[\rm (i)]  $\mathcal{B}_3 = \{K_3 = C_3\}$,
\item[\rm (ii)]  $\mathcal{B}_4 = \{K_4, K_{2,s}, K'_{2,s}, s \geq 2\}$. 
\item[\rm (iii)]  $\mathcal{B}_5 = \{\bar C^2_5, C^3_5, \bar C^3_5, C^4_5, K_5, K^+_{2,l}, l \geq 3, P_{r,s}, P'_{r,s}, P''_{r,s}, r \geq 1, s \geq 1\}$
\end{enumerate}
\end{theorem}

\begin{proof}
\textbf{Case (i)}.This case is easy to see.

\textbf{Case (ii)}. First observe, that deleting any edge or any vertex cannot increase the circumference of a graph.

Suppose that there is a $2$-connected graph $B$ with $\mathrm{cir}(B) = 4$ such that $B \notin \mathcal{B}_4$. Let $B$ have the smallest order and then the smallest size among all counterexamples.

Let $B$ have a vertex $v$ with $\deg_B(v) = 2$. If both neighbours $w_1$ and $w_2$ of $v$ are of degree at least 3, then the subgraph $B - v \in \mathcal{B}_4$. But then either $B \in \mathcal{B}_4$ or $B$ contains  a cycle $C_k$, $k \geq 5$. A contradiction.
If $v$ has a neighbour $w$ of degree 2, and if $x_v$ and $x_w$ are other neighbours of $v$ and $w$, respectively, we replace the $3$-path $x_vvwx_w$  with the edge $x_vx_w$. The resulting graph $\bar B$ contains a $2$-cycle and $B$ is $C_4 = K_{2, 2}$, a contradiction.

Let the minimum degree $\delta(B) \geq 3$. Consider a vertex $v \in V(B)$ with $\deg_B(v) = \delta(B)$. Then the subgraph 
$B - v$ is connected and each non-trivial block of it is in $\mathcal{B}_3 \cup \mathcal{B}_4$. One can easily see that then $B$ contains a cycle of length at least 5, a contradiction.

\textbf{Case (iii).} Assume that there is a $2$-connected graph $B$ with $\mathrm{cir}(B) = 5$ and $B \notin \mathcal{B}_5$. Let $B$ have the smallest order and then the smallest size among all counterexamples. Next we consider two sub-cases.
 
\textbf{Subcase 1}. Let $B$ have a vertex $v$ with $\deg_B(v) = 2$. If both neighbours of $v$ have degree at least 3, then the subgraph $B - v$ is connected and has all its non-trivial blocks in $\mathcal{B}_3 \cup \mathcal{B}_4 \cup \mathcal{B}_5$. If at least one neighbour $w$ of $v$ is of degree 2 and if $x_v$ and $x_w$ are the other neighbours of $v$ and $w$, respectively, we replace the $3$-path $x_vvwx_w$  with the edge $x_vx_w$. The resulting graph $\bar B$ contains a $3$-cycle and $B \in \mathcal{B}_5$ or $B$ contains a cycle of length at least 6, a contradiction.

\textbf{Subcase 2.} 
Let the minimum degree $\delta(B) \geq 3$. Consider a vertex $v \in V(B)$ with $\deg_B(v) = \delta(B)$. Then the subgraph 
$B - v$ is connected and each non-trivial block of it is in $\mathcal{B}_3 \cup \mathcal{B}_4 \cup \mathcal{B}_5$. One can easily see that then $B$ contains a cycle of length at least 6, a contradiction.
\end{proof}


\section{Loose edge-connection of $2$-connected graphs}

The complete solution for $2$-connected graphs is given in the following.

\begin{theorem}\label{thm:in}
If $G$ is a $2$-connected graph, then $\mathrm{lec}(G) = \min\{3, \mathrm{rc}(G)\}.$
\end{theorem}

\bigskip
For its proof we first show the following lemma.

\bigskip
\begin{lemma}\label{le:2c}
If $G$ is $2$-connected, then $\lec(G) \leq 3$.
\end{lemma}

\begin{proof}
We have to show how to colour the edges of the graph $G$ with colours from the set $A = \{a, b, c\}$ to get a loose edge-connection of $G$. Our procedure depends on the circumference $\mathrm{cir}(G)$ of $G$.

If $\mathrm{cir}(G) = 3$, then we have to colour the graph $K_3$, see Theorem \ref{thm:1}, the complete graph on three vertices. It will play an important role later. It is easy to see that $\lec(K_3)=1$. 

 If $\mathrm{cir}(G) = 4$, then we have to color the graphs
from the set $\mathcal{B}_4$, see Theorem \ref{thm:1}.

Let the loose edge-colouring of $K_{2, r}$ be as follows: $\varphi(xv_1) = a$, $\varphi(xv_2) =  b$, $\varphi(xv_i) = c, i \in [3, r]$, $\varphi(yv_1) = b$, $\varphi(yv_2) = c$ and $\varphi(yv_i,) = a, i \in [3, r]$. Observe, that between any two vertices of $K_{2, r}$, except for the pair $x$ and $y$, there is a loose (i.e. at least three coloured) path. For the pair $x$ and $y$, if $r \geq 3$, there are three different bi-chromatic $x, y$-paths of length three using the same three colours.

A loose edge-colouring of the graph $K'_{2, r}$ is obtained from that of $K_{2, r}$ using $\varphi(xy) = a$.

 If $\mathrm{cir}(G) = 5$, then we have to color the graphs
from the set $\mathcal{B}_5$, see Theorem \ref{thm:1}.
Let the loose edge colouring of $P_{r,s}$ be defined as follows:
$\varphi(xu_1) = a$, $\varphi(xu_i) = b$, $i \in [2,r]$, $\varphi(yz) = c$, $\varphi(xv_1) = c$, $\varphi(xv_j) = a, j\in[2,s]$, $\varphi(yu_1)=b$, $\varphi(yu_i)=c$, $i \in [2, r]$, $\varphi(zv_1) = a$, $\varphi(zv_j) = b$, $j \in [2, s]$.

Note that for $r \geq 2$, $s \geq 1$, and any pair of vertices of $P_{r,s}$, except for the pair $x$ and $y$, there is, in $P_{r,s}$, a loose path using three colours. For the pair $x$ and $y$ there are three differently bi-colored $x, y$-paths of length two using the same three colours. The graph $P'_{r,s}$ ($P''_{r,s}$), $r \geq 1, s \geq 1$, comes from the graph $P_{r,s}$ by inserting the edge 
$xy$ (the edges $xy$ and $xz$, respectively). The inserted edge in $P'_{r,s}$ is coloured with colour $\varphi(xy) = b$ (The inserted edges in $P''_{r,s}$ are colored with colours $\varphi(xy) = c$ and $\varphi(xz) = b$).

A loose edge-colouring of the graph $K^+_{2,r}$ is obtained from that one of $K_{2,r}$ using colour $a$ for the edge $v_2v_3$.

The proof that the loose edge-connectivity of the remaining graphs from $\mathcal{B}_5$ is at most 3 is easy and is left to the reader.

 If $\mathrm{cir}(G) \geq 6$, we proceed as follows.
Colour the edges of the longest cycle $C$ alternatively by $1,2,3,1,2,3,..$ and other edges of $G$ with colour 1. We distinguish three cases.

\textbf{Case 1.}. For any two vertices $x$ and $y$ on the cycle $C$ there is a path of length at least three, which is a loose path.

\textbf{Case 2.} If $x\in V(C)$ and $y\notin V(C)$, then there are two paths connecting $y$ with two vertices $z_1,z_2$ on the cycle $C$. At least one of them is distinct from $x$, say $z_1$. By Case 1 there is a path $L$ on $C$ of length at least 3 from $z_1$ to $x$, which is loose (uses three colors). Now, taking the path from $z_1$ to $y$ (which is vertex disjoint with $L$ except of vertex $z_1$) we obtain a loose $x, y$-path.

\textbf{Case 3.} Let $x, y \notin V(C)$. Similarly as above, there are two vertex disjoint paths from $x$ to $z_1$ and from $y$ to $z_2$
with $z_1, z_2 \in V(C), y\neq z_2$. As in Case 2 there is a loose path from $z_1$ to $z_2$ along the cycle $C$ which leads to a loose path from $x$ to $y$.
\end{proof}

\begin{cor}\label{co:1}
For every two vertices $x$ and $y$ of a $2$-connected graph $G$ with $\cir(G) \geq 6$ there is a loose $x, y$-path.
\end{cor}

\begin{proof}
 It follows directly from the proof of Lemma \ref{le:2c}.
\end{proof}

\begin{proof} Theorem \ref{thm:in}. To prove it  first observe the following statements.
\begin{enumerate}
    \item 
   $\mathrm{lec}(G) = 1$  if and only if $\mathrm{rc}(G) = 1$ if and only if $G$ is complete graph.
   \item
    $\mathrm{lec}(G) = 2$  if and only if $\mathrm{rc}(G) = 2$.
\end{enumerate}
Now for all remaining $2$-connected graphs we have 
 $$3 = \mathrm{lec}(G) \leq \mathrm{rc}(G).$$
 This proves the theorem.
\end{proof}

\section{Loose edge-connection number of complete bipartite graphs}

 The complete solution for the complete bipartite graphs is given in the following.

\begin{theorem}\label{thm:bip}
If $K_{r,s}, r \geq s \geq 1$, is a complete bipartite graph, 
then the following statements hold:
\begin{enumerate}
\item[\rm (i)]
$\mathrm{lec}(K_{1,s}) = s.$
\item[\rm (ii)] 
$\mathrm{lec}(K_{r,s}) = 2$ if $2 \leq s \leq r \leq 2^s.$
\item[\rm (iii)] 
$\mathrm{lec}(K_{r,s}) = 3$ otherwise. 
\end{enumerate}
\end{theorem}

\begin{proof}
 Let $K_{r,s}$ be a complete bipartite graph with the vertex set $V(K_{r,s}) = X \cup Y$, where $X = \{x_1, \dots, x_r\}$ and $Y = \{y_1, \dots, y_s\}$. Let $S_i = K_{1,s}$ be the star with the vertex set $V(S_i) = \{x_i\} \cup Y$.
 
 First we consider case (ii).
 
 Suppose that $\phi$ is an edge-colouring of $K_{r,s}$ with two colours $a$ and $b$. Recall that a path in an edge-coloured graph is called {\it{bi-chromatic}} (resp. {\it{monochromatic}}) if on its edges exactly two colors (resp. one) are (is) used.
 
 Let $\mathbf{s}(i) = (a_{i,1}, \dots, a_{i,s})$ be a vector of colours of the edges of the star $S_i$ such that $\phi(x_iy_j) = a_{i,j}$ for all $i \in [1, r]$ and $j \in [1, s]$.
 
 This colouring has the following obvious two properties:

\begin{claim}\label{cl:b1}
A bi-chromatic $x_k, x_l$-path exists in  $K_{r,s}$ if and only if  $\mathbf{s}(k) \ne \mathbf{s}(l)$. 
\end{claim}

\begin{claim}\label{cl:b2}
If there is $a_{i,k} \ne a_{i,l}$ in $\mathbf{s}(i) = (a_{i,1}, \dots, a_{i,s})$, then $K_{r,s}$ contains a bi-chromatic $2$-path $y_kx_iy_l$. 
\end{claim}

Observe, that for $r > 2^s$ there is no loose edge-coloring of $K_{r,s}$ with two colours. The reason is that the number of different binary vectors of dimension $s$ is $2^s$. Then, by the pigeonhole principle, there is a pair of stars, $S_k$ and $S_l$, with the same vector of colours, $\mathbf{s}(k)$ and $\mathbf{s}(l)$, respectively, and, by Claim \ref{cl:b1}, we have no bi-chromatic path of length two between the pair of vertices $x_k$ and $x_l$.

If $2 \leq s \leq r \leq 2^s$, then there exists a loose edge-colouring of $K_{r,s}$. It is enough to consider the vectors of colours for the stars $S_i, i \in [1, s - 1]$, with the following properties: $\phi(x_iy_j) = a$ for $j \leq i$ and $\phi(x_iy_j) = b$ for $i + 1 \leq j \leq n$.
For $i \in [s, r]$ we colour the edges of stars $S_i$ with such vectors of colours that are mutually distinct from  already chosen ones. 

To see that this edge-colouring is loose, observe that Claim \ref{cl:b1} ensures us that for any pair of vertices $x_k$ and $x_l$, $1 \leq k < l \leq r$, there is a bi-chromatic path of length two between them.
The choice of the first $s - 1$ vectors of colour ensures the existence of a bi-chromatic $y_k,y_l$-path of length two for any pair of vertices $y_k$ and $y_l$ for any pair $k$ and $l$ with $1 \leq k < l \leq n$. See Claim \ref{cl:b2}.  

This completes the proof of the theorem for the case (ii).

The proof of theorem in the case (i) follows from Lemma \ref{le:tre} and in the case (iii) from Lemma \ref{le:2c}.
\end{proof}


\section{Graphs with diameter at most 2}
To continue we need some more definitions. 

The \textit{join} of two simple graphs $G$ and $H$, written $G \vee H$, is the graph obtained from the disjoint union 
$G + H$ by adding the edges $\{xy: x \in V(G), y \in V(H)\}$,
see \cite{Wes}.

A nontrivial block $B$ of a graph $G$ is called $\textit{large}$ if $\cir(B) \geq 6$ and it is called
$\textit{small}$ if $2 \leq \cir(B) \leq 5$. 

Recall that a block $B$ is called \textit{trivial} if it consists of a cut-edge $e$. In this case we sometimes will write $B = e$.
 
It is easy to see that graphs of $\diam(G) \geq 3$ have the loose edge-connection number $\mathrm{lec}(G) \geq 3$. So the necessary condition for graphs $G$ to have  $\mathrm{lec}(G) = 2$ is to have $\diam(G) = 2$. 
In the previous section we have characterized all complete bipartite graphs $B$ that have $\mathrm{lec}(B) = 2$.

Observe, that if a graph $H$ is a spanning subgraph of a non-complete graph $G$ and $\mathrm{lec}(H) = 2$, then $\mathrm{lec}(G) = 2$

Let $G$ and $H$ be graphs with loose edge-connection number two. Then it is easy to prove that $\mathrm{lec}(G \vee H) = 2$, and $\mathrm{lec}(K_1 \vee (G + H)) = 2$.

In Theorem \ref{thm:in} we have mentioned that for a graph $G$ it holds $\mathrm{lec}(G) = 2$  if and only if $\mathrm{rc}(G) = 2$. Several classes of graphs $G$ with $\mathrm{rc}(G) = 2$ are known, see e.g. \cite{CLRTY}, \cite{Ne:05}, or \cite{KeSch11}. For example, Caro et al. \cite{CLRTY} have proved that any non-complete $n$-vertex graph with minimum degree $$\delta(G) \geq \frac{n}{2} + \log n$$ has $\mathrm{rc}(G) = 2$. 

It is known (cf. \cite{ChFMY11}) that deciding whether $\mathrm{rc}(G) = 2$ is an NP-complete problem. Hence, the problem to decide whether  $\mathrm{lec}(G) = 2$ is also an NP-complete. 

\bigskip
The following problem seems to be interesting.

\begin{problem}
Characterize all connected graphs $G$ with 
$$\mathrm{lec}(G) = 2.$$
\end{problem}

\begin{theorem}\label{thm:dia2}
Let $G$ be a graph with $\diam(G)\leq 2$. Then 
\begin{enumerate}
\item[\rm (i)] $\lec(G)=1$ if $G = K_n, n \geq 2$, and
\item[\rm (ii)] $2 \leq \lec(G) \leq \max\{3, \Delta(C(G))\}$  otherwise. Moreover, all three bounds are tight.
\end{enumerate}
\end{theorem}

\begin{proof}
Case (i) is easy. 

If $\diam(G)=2$, then G is either $2$-connected or has exactly one cut-vertex. If it is $2$-connected, then the proof of our theorem follows from Lemma \ref{le:2c}. 

If $G$ contains a cut-vertex, $v$, then each block of G is trivial or $2$-connected. The vertex $v$ is adjacent to all vertices of $V(B)-v$ because of $\diam(G)= 2$. 
Trivial blocks of $G$ form the graph $C(G)$ which is a star $K_{1,\Delta(C(G))}$. Each nontrivial block $B$ is either a large block or contains a small block as a spanning subgraph. 

The trivial blocks are coloured with distinct colours from   $[1, \Delta(C(G)]$. The remaining blocks are coloured as in Lemma \ref{co:1} with colours from $[1, 3]$.

Next, we need to show that for any pair of two vertices $u$ and $w$ there is either an edge $uw$, or a bi-chromatic $u, w$-path of length two, or a loose $u,w$-path.

If the vertices $w$ and $z$ belong to the same block $B$ we are done. 

Let $u \in V(B_i)$ and $w \in V(B_j)$, where $B_i$ and $B_j$ are different non-trivial blocks of $G$. If there is either a loose $u, v$-path in $B_i$ or a loose $w, v$-path in $B_j$ then these paths together with the edge $vu$ or $vw$, respectively, give a loose $u, w$-path. This situation always appears, with exceptions, when both vertices $u$ and $w$ 
correspond to the vertices $x, y$, resp. $z$ of the blocks (graphs) from the families $\mathcal{B}_3, \mathcal{B}_4$, and $\mathcal{B}_5$.
But also in these exceptional cases one can easily find a needed $u, w$-path.

To see, that the bounds $2$ and $3$ are tight, see Theorem \ref{thm:bip}. For the tightness of the third bound see the graph $K_{1,s}, s \geq 3$. 
\end{proof}

\section{Graphs with diameter at least 3}

A nontrivial block $B$ of a graph $G$ is called to be \textit{of type $P$} (or \textit{of type $Q$}, or \textit{of type $R$}) if it is of type $(t, P)$ (or of type $(t, Q)$, or of type $(t, R)$, respectively) for some $t \geq 0$.

We recall that a block $B$ of a graph $G$ is small if $B \in \{\mathcal{B}_3 \cup \mathcal{B}_4 \cup \mathcal{B}_5\}$
 
A cut-vertex $v$ of an edge-coloured block $B$ of a graph $G$ is called to be of $\textit{type}$ $(a,b)$ with respect to $B$ if $deg_B(v)=2$, one edge incident with $v$ is coloured with $a$ and the second one with $b$, 
$a \neq b$, $B \in \mathcal{B}_3 \cup \mathcal{B}_4 \cup \mathcal{B}_5$. Otherwise it is called $\textit{universal}$ with respect to $B$.

\begin{theorem}\label{thm:dia3}
Let $G$ be a connected graph with $\diam(G)\geq 3$. The following statements hold:
\begin{enumerate}
\item[\rm (i)] If $G$ is $2$-connected or
$0 \leq \Delta(C(G))\leq 2$, then $3 \leq \lec(G) \leq 4$.  

Furthermore, $\lec(G) = 4$ if and only if $G$ contains a block $B$ which is a cycle of length 4 having three distinct vertices $b_1, b_2, b_3 \in V(B)$ with $\deg_{C(G)}(b_1) = \deg_{C(G)}(b_2) =  \deg_{C(G)}(b_3) = \mathrm{rw}(C(G)) = 2$, or
\item[\rm (ii)] If $\Delta(C(G)) \geq 3$,  then $\mathrm{rw}(C(G))\leq \lec(G) \leq \mathrm{rw}(C(G)) + 1$.\\
Furthermore, $\lec(G) = \mathrm{rw}(C(G)) + 1$ if and only if

\bigskip
\begin{itemize}
    \item[\rm (1)] 
    $G$ contains a block $B$ which is a cycle of length 5 having two distinct vertices $b_1, b_2 \in V(B)$ with $\deg_{C(G)}(b_1) = \deg_{C(G)}(b_2) = \mathrm{rw}(C(G)) = 3$, or
    \item[\rm (2)]
    $G$ contains a block $B$ which is a cycle of length 4 and $\deg_{C(G)}(b) = \mathrm{rw}(C(G)) = 3$ for a vertex $b \in V(B)$, or
    \item[\rm (3)]
    $G$ contains a block $B$ which is a cycle of length 3 having  two distinct vertices $b_1, b_2 \in V(B)$ with 
    $\deg_{C(G)}(b_1) + \deg_{C(G)}(b_2) \geq  2 \mathrm{rw}(C(G)) - 1$, or
    \item[\rm (4)]
    $G$ contains a block $B$ which is a diamond $D$, $V(D) = \{x, v, y, w\}$, $\mathrm{rw}(C(G)) = 3$, and 
    $\deg_{C(G)}(x) + \deg_{C(G)}(y) = 6$ or 
    $\sum_{b \in V(B)}\deg_{C(G)}(b) \geq 11$.
\end{itemize}
\end{enumerate}
\end{theorem}

\begin{proof}
{\bf{Part 1. Colouring.}} Let $\mathcal{D} = \{K_3, C_4, D, K_4, C_5, C_5^1, C_5^2, \bar C_5^2, C_5^3, \bar C_5^3, C_5^4, K_5\}$ be a set of specific small blocks. Recall that any large block and any small block, other than that from $\mathcal{D}$, has a loose edge-colouring with three colours. Between any two distinct vertices of large blocks and any two vertices of small blocks other than that from $\mathcal{D}$, except two specific $x$ and $y$ of small blocks, there exists a loose path (Recall, that a path is loose if has length at least three, which edges are coloured with at least three colours). The vertices $x$ and $y$ are connected with three mutually distinct bi-chromatic $x, y$-path. See the proof of Lemma \ref{le:2c}.

Our colouring follows the properties of the block-cutpoint graph $B(G)$ of $G$. Choose a vertex of $B(G)$ corresponding to a block $B$ as a root. We colour edges of the blocks from this root to the leaf blocks. First we colour edges of $B$ as described in Lemma \ref{le:2c} if $B$ is $2$-connected and with colour 1 if $B$ is trivial. If $B$ is trivial (i.e. $B = \{e\}$, where $e$ is an edge) we colour all edges of the component of $C(G)$ containing the edge $e$ according to the method used in the proof of Lemma \ref{le:tre}.

General Step: Let $v$ be a cut vertex on $B$ and let $B_1, \dots, B_k, S_1,\dots,S_l, T_1, \dots,T_r$ be other blocks containing $v$ as cut-vertex, where $B_i$, for $i \in [1,k]$, is a large block, $S_j$, for $j \in [1,l]$, is a small block, and  $T_s$, for $s \in [1,r]$, is a trivial block.
First we colour large blocks, next small blocks different from blocks from $\mathcal{D}$, then small blocks from $\mathcal{D}$, and, finally, trivial blocks.

We distinguish three cases.

\textbf{Case 1.} Let $deg_B(v)=1$ and $\varphi(wv) = a$ be a colour of the unique edge $wv$ in $B$. Then for every
$i \in [1,k]$ the edges of the block $B_i$ are coloured with colours from the color set $\{a, b, c\}$ as stated in Corollary \ref{co:1}.

If $S_j$, $j \in [1,l]$, is a small block from $(\mathcal{B}_3 \cup \mathcal{B}_4 \cup \mathcal{B}_5)\setminus \mathcal{D}$ and $\deg_{S_j}(v) = 2$, then the edges incident to $v$ in $S_j$ are colored with colours from $\{a, b, c\}$ so that the vertex $v$ becomes of type $(b,c)$, where $b \ne a \ne c$. Next we finish the colouring of edges of $S_j$ as described in the proof of Lemma \ref{le:2c}.
If $\deg_{S_j}(v) \geq 3$, then we colour $S_j$ directly as in Lemma \ref{le:2c}.

If $S_j \in \mathcal{D}$, then its edges are coloured together with all trivial blocks (edges) incident to the vertices of the block $S_j$. In this case we colour the subgraph $H$ of type $(t, O)$ for $t \geq 1$ and $O \in \{P, Q, R\}$ as described in Lemmas \ref{le:P_t}, \ref{le:Q_t}, and \ref{le:R_t}, respectively, taking into consideration the following three facts: (i) The edges of the longest cycle of $H$ incident with the vertex $v$ have to obtain in $H$ the colours $b$ and $c$. (ii) Some of edges of $H$ at $v$ can already be coloured (This can happen if the block $B$ is trivial, or there are several blocks of types from $\{P, Q, R\}$ at the considered cut vertex $v$). (iii) If $H$ is of type $(1, R)$, we colour its edges with three colours as in $R_1$ because $H \subset R_1$.   

If $S_j \in \mathcal{D}$ and there no trivial blocks incident to vertices of $S_j$, i.e. $S_j$ is subgraph of type $(0, O)$, we colour its edges as in Lemma \ref{le:2c} except of the case when $S_j = C_3 = K_3$. In the later case the edges of $C_3$ are coloured with three distinct colours. 
 
Let there be trivial blocks at $v$, i.e., $r \geq 1$. 
If this block is a part of a subgraph of type $O \in \{P, Q, R\}$ at $v$, then all trivial blocks are coloured as a part of the colouring of the first considered subgraph of type $O$.
If there is no subgraph of type $O \in \{P, Q, R\}$ at $v$, then
the edges of all trivial blocks (i.e. edges) are part of a component of $C(G)$. In this case we colour all edges of this component as described in the proof of Lemma \ref{le:tre}.

\textbf{Case 2.} Let $deg_B(v)=2$ and the vertex $v$ be of type $(a,b)$. Then the large block $B_i$, for any $i \in [1,k]$ is coloured with colours from the color set $\{a, b, c\}$ as in the proof of Lemma \ref{le:2c}.

The block $S_j$, for any $j \in [1,l]$, $S_j \notin \mathcal{D}$, is handled as a large block if $v$ has, in it, degree at least three, or such that the vertex $v$ becomes of type $(a, b)$ in the colouring of $S_j$, if $v$ has degree two. If $S_j \in \mathcal{D}$ is a subgraph $H$ of type $O \in \{P, Q, R\}$, then we consider the subgraph $H$ of type $O$ and colour it as described in Case 1.

Trivial blocks $T_s$ are handled in the same way as in Case 1.

\textbf{Case 3.} Let $deg_B(v)\geq 3$. Our colouring in this case continues as in Case 2 keeping in mind that all small blocks $S_j$, $j \in [1,l]$, have to be colored in such a way that the cut-vertex $v$ becomes of type $(a, b)$ if it is of degree two in $S_j$.

If all blocks containing $v$ as a cut-vertex are colored, $v$ is called $\textit{ready}$. 

To continue in the colouring of $G$ a next, not ready, cut-vertex from a colored block is chosen.

The procedure of colouring finishes if every cut vertex of $G$ is ready.

Observe, see Lemmas \ref{le:tre}, \ref{le:R_t}, \ref{le:Q_t}, \ref{le:P_t} and \ref{le:2c}, that our colouring uses colours from the set $M = [1,\mathrm{rw}(C(G))]$ except for the cases listed in our Theorem \ref{thm:dia3}. These exceptional cases use the colours from the colour set $M^* = [1, 4]$ if $G$ contains a subgraph $H$ of type $Q$ described in Lemma \ref{le:Q_t} (i) and from the colour set $M' = [1,\mathrm{rw}(C(G)) +1]$ if $G$ contains subgraphs $H$ of type $O \in \{P, Q, R\}$ described in the Lemmas \ref{le:P_t} (ii), \ref{le:Q_t} (ii), and \ref{le:R_t} (iii), respectively.

{\bf{Part 2. Loose connectivity}}. To finish the proof of the theorem it is necessary to show that any two distinct vertices $w_1$ and $w_2$ are connected by an edge $w_1w_2$, or bi-chromatic $w_1,w_2$-path of length 2, or a loose (with at least three different colours) $w_1,w_2$-path of length at least 3. Such a path exists if both $w_1$ and $w_2$ belong to the same block, to the same component of $C(G)$, or to the same subgraph $H$ of type $O \in \{P, Q, R\}$. 

If there is an $w_1,w_2$-path going through two vertices, say $u_1$ and $u_2$, of a large block or of a small block, then the part between them can be replaced by a loose $u_1,u_2$-path lying inside the block.

Hence, it is sufficient to show the existence of such a loose $w_1,w_2$-path for any two blocks $D_1$ and $D_2$ with a common cut-vertex $z$, $w_1\in V(D_1)$ and $w_2\in V(D_2)$, $w_1\neq z\neq w_2$.

If the block $D_1$ is large, then there is a loose $w_1, z$-path in $D_1$, which together with a $z,w_2$-path in $D_2$ gives a needed loose $w_1, w_2$-path.

Let the block $D_1$ be small and $D_2$ be either small or trivial. If the vertex $w_1$ is not corresponding to any of the vertices $x$ and $y$ of $D_1 \in \{K_{2,s}, K'_{2,s}, K^+_{2,s}\}$ or to any of the vertices $x, y$ and $z$ of $D_1 \in \{P_{2,s}, P'_{2,s}, P''_{2,s}\}$, respectively, then 
there is, in $D_1$, a loose $w_1,z$-path. This path together with a $z, w_2$-path gives a required loose $w_1,w_2$-path.

If the vertex $w_1$ and the vertex $w_2$ are corresponding to some of the above mentioned vertices $x, y$, or $z$ of $D_1$ and of $D_2$, respectively, Then one can easy recognise in $G$ a loose $w_1, w_2$-path of length 3 or 4.

\end{proof}

\bigskip
\noindent{\bf Declaration of competing interest}

The authors declare that they have no known competing financial interests or personal relationships that could have appeared to influence the work reported in this paper.

\medskip
\noindent{\bf Acknowledgements}

The research of all three authors was supported in part by the DAAD-PPP project "Colourings and connection in graphs" with project-ID 57210296.

The work of the second author was supported by the Slovak Research and Development Agency under the contract No. APVV-19-0153 and by the Slovak VEGA grant 1/0574/21.


\begin{thebibliography}{2}

\bibitem{ABBFKS} S. A. van Aardt, C. Brause, A. P. Burger, M. Frick, A. Kemnitz, and I. Schiermeyer, \emph{Proper connection and size of graphs},  Discrete Math. 340 (2017) 2673--2677.

\bibitem{borozan12} V. Borozan, S. Fujita, A. Gerek, C. Magnant, Y. Manoussakis, L. Montero, and Z. Tuza,
\emph{Proper connection of graphs}, Discrete Math. 312 (2012) 2550--2560.

\bibitem{BJSch} Ch. Brause, S. Jendro\ml, I. Schiermeyer, \emph{Odd connection and odd vertex connection of graphs}, Discrete Math. 341 (2018) 3500-3512.

\bibitem{BJSch-21} Ch. Brause, S. Jendro\ml, I. Schiermeyer, \emph{From colourful to rainbow paths in graphs: Colouring the vertices}, Graphs Combin. 37 (2021) 1823--1839

\bibitem{CLRTY} Y. Caro, A. Lev, Y. Roditty, Z. Tuza, R. Yuster, \emph{On the rainbow connection}, Electron. J. Combin. 15 (2008) R57.

\bibitem{ChFMY11} S. Chakraborty, E. Fischer, A. Matsliak, R. Yuster, \emph{Hardness and algorithms for rainbow connectivity}, J. Combin. Optim. 21 (2011) 330--347.

\bibitem{Chan}
L. S. Chandran, A. Das, D. Rajendraprasad, and N. M. Varma, \emph{Rainbow connection number and connected dominating sets}, J. Graph Theory 71 (2012) 206--218.


\bibitem{CJMZ} G. Chartrand, G. L. Johns, K. A. McKeon, and P. Zhang, \emph{Rainbow connection in graphs}, Math. Bohemica 133 (2008) 85--98.

\bibitem{Cz:12} J. Czap, \emph{Edge looseness of plane graphs}, Ars Math. Contemp. 9 (2015) 285--296.

\bibitem{CJ} J. Czap and S. Jendro\ml, \emph{Facially-constrained colorings of plane graphs: A survey}, Discrete Math., 340 (2017) 2691--2703.


\bibitem{Czap} J. Czap, S. Jendro\ml, and J. Valiska, \emph{Conflict-free  connections  of  graphs}, Discuss. Math. Graph Theory 38 (2018) 911--920.


\bibitem{HRS} P. Holub, Z. Ryj{\'a}\v{c}ek, and I. Schiermeyer, \emph{On forbidden subgraphs and rainbow connection in graphs with minimum degree $2$}, Discrete Math. 338 (2015) 1--8.

\bibitem{HRSV1} P. Holub, Z. Ryj{\'a}\v{c}ek, I. Schiermeyer, and P. Vrana, \emph{Rainbow connection and forbidden subgraphs}, Discrete Math. 338 (2015) 1706--1713.

\bibitem{HRSV2} P. Holub, Z. Ryj{\'a}\v{c}ek, I. Schiermeyer, and P. Vrana, \emph{Characterizing forbidden pairs for rainbow connection in graphs with minimum degree $2$}, Discrete Math. 339 (2016) 1058--1068.

\bibitem{LSS} X. Li, Y. Shi, and Y. Sun, \emph{Rainbow connections of graphs: A survey}, Graphs Combin. 29 (2013) 1--38.

\bibitem{KeSch11} A. Kemnitz, I. Schiermeyer, \emph{Graphs with rainbow connection number two}, Discuss. Math. Graph Theory 31 (2011) 313--320.

\bibitem{Ne:05} S. Negami, \emph{Looseness ranges of triangulations on closed surfaces}, Discrete Math. 303 (2005) 167--174.


\bibitem{Wes} D. B. West, \emph{Introduction to Graph Theory}, Prentice Hall, 2000. 


\end{thebibliography}
\end{document}